\documentclass[11pt]{article}
\usepackage{amssymb,amsmath,latexsym,amsthm}
\usepackage{color,mathrsfs}
\usepackage{url}
\usepackage{enumerate}

\usepackage{endnotes}

\usepackage{supertabular}
\usepackage[dvips]{graphics}
\usepackage[xetex]{graphicx}
\usepackage{fancyhdr}
\usepackage{pstricks,pst-plot,pst-node,pstricks-add}

\title{Minimization of energy per particle among Bravais lattices in $\R^2$ : Lennard-Jones and Thomas-Fermi cases}

\author{Laurent B\'{e}termin \and Peng Zhang}

\newtheorem{thm}{Theorem}[section]

\newtheorem{corollary}[thm]{Corollary}
\newtheorem{prop}[thm]{Proposition}

\theoremstyle{definition}
\newtheorem{remark}[thm]{Remark}

\newcommand{\R}{\mathbb R}
\newcommand{\Z}{\mathbb Z}
\newcommand{\N}{\mathbb N}
\numberwithin{equation}{section}

\allowdisplaybreaks

\begin{document}
\maketitle
\begin{abstract}
We prove in this paper that the minimizer of Lennard-Jones energy per particle among Bravais lattices is a triangular lattice, i.e. composed of equilateral triangles, in $\R^2$ for large density of points, while it is false for sufficiently small density. We show some characterization results for the global minimizer of this energy and finally we also prove that the minimizer of the Thomas-Fermi energy per particle in $\R^2$ among Bravais lattices with fixed density is triangular.
\end{abstract}

\section{Introduction}
Understanding the structure of matter at low temperature has been a challenge for many years. In this case, one of the simplest models is to consider identical points as particles interacting in a Lennard-Jones potential. This model is deterministic, therefore we do not consider either entropy nor other quantum effects. The problem is to find the configuration of the points which minimize the total interaction energy, called the Lennard-Jones energy. Radin, in \cite{Rad1}, studied this problem in one dimension and showed that, in the case of infinite points, the minimizer is periodic. His method is not adaptable in higher dimensions and he studied, in \cite{Rad2,Rad3} the case of short range interactions and proved the first result of crystallization in two dimensions for a hard-sphere model. In the meantime, Ventevogel and Nijboer gave in \cite{VN1,VN2,VN3} more general results in one dimension for Lennard-Jones energy per particle. Indeed, they showed that a unique lattice of the form $a_0\N$ minimizes the Lennard-Jones energy and that all lattices $a\N$ with $a\leq a_0$ minimize this energy when the density of points $\rho=a^{-1}$ is fixed. Our paper gives some results in the spirit of the latter paper.\\ \\
After a numerical investigation of Yedder, Blanc, Le Bris, in \cite{YBLB}, about the minimization of the Lennard-Jones and the Thomas-Fermi energy in $\R^2$, it seemed that the triangular lattice, also called ``hexagonal lattice" -- which is composed of equilateral triangles -- is the minimum configuration for Lennard-Jones energy among any lattices and for Thomas-Fermi energy with nuclei density fixed. Some time after, Theil, in \cite{Crystal}, gave the first proof of crystallization in two dimensions for a ``Lennard-Jones like" potential, with a minimum less than one but very close to one and long range interaction. He showed that the global minimizer of the total energy is triangular. His method was adapted by E and Li, in \cite{ELi}, for a three-body potential with long range interactions in order to obtain a honeycomb lattice as global minimizer -- see also the works of Mainini, Piovano and Stefanelli in \cite{Stef1,Stef2} about the crystallization in square and honeycomb lattices for three-body potentials with short range interactions -- and by Theil and Flatley in three dimensions in \cite{TheilFlatley}.\\ \\
Furthermore Montgomery, in \cite{Mont}, proved that the triangular lattice is the unique minimizer of theta functions among Bravais lattices with fixed density and hence the  unique minimizer of the Epstein zeta function, thanks to the link between these two functions. As the Lennard-Jones potential is a linear sum of Epstein zeta functions, it is natural to study the problem of minimization of the Lennard-Jones energy among Bravais lattices with and without fixed density. However, there are few results about minimization in the general case of periodic systems. For example, Cohn and Kumar described in \cite{CohnKumar} a method and a conjecture for completely monotonic functions. It is interesting to observe that this kind of problem is connected with the theory of spherical design due to Delsarte, Goethals and Seidel in \cite{DelGoethSeid} and linked to the layers of a lattice, among others, by Venkov and Bachoc in \cite{Venkov1,BachocVenkov} and by Coulangeon et al. in \cite{Coulangeon:kx,Coulangeon:2010uq,CoulLazzarini}. \\ \\
    In this paper, our main results are :\\
    
    \noindent \textbf{Theorem:}
    \textit{ \begin{itemize} 
    \item Let $V_{LJ}(r)=r^{-12}-2r^{-6}$ be the Lennard-Jones potential, then the minimizer of the energy $\displaystyle E_{LJ}(L)=\sum_{x\in L\backslash\{0\}}V_{LJ}(\|x\|)$ among all Bravais  lattices of $\R^2$ with fixed density sufficiently large is triangular and unique, up to rotation.\\ 
    \item A minimizer of $E_{LJ}$ among all Bravais lattices with fixed density sufficiently small cannot be triangular.\\ 
    \item Let $W_{TF}:\R_+^* \to \R$ be the solution of $-\Delta h+\pi h=\delta_0$ which goes to $0$ at infinity, then the minimizer of the Thomas-Fermi energy $\displaystyle E_{TF}(L)=\displaystyle\sum_{x\in L\backslash\{0\}}W_{TF}(\|x\|)$ among all Bravais lattices of $\R^2$ with density fixed is triangular and unique, up to rotation. \\\end{itemize}}
    \noindent This paper is structured as follows : in Section 2, we introduce the notations; in Section 3, we show that the minimizer of the Lennard-Jones energy per particle among Bravais lattices with fixed density, if the density is sufficiently large, it is triangular and unique. Moreover we give numerical results and a conjecture for the minimization with density fixed and we have arguments in order to explain why the global minimizer, among Bravais lattices without fixed density, is triangular; in Section 4, we use proof of Blanc in \cite{BL1} to find a lower bound for the interparticle distance of the global minimizer, and finally in Section 5 we study the same kind of problem for the Thomas-Fermi model only when the density is fixed and we prove that the triangular lattice is the unique minimizer of the Thomas-Fermi energy per particle in $\R^2$.

\section{Preliminaries}
A Bravais lattice (also called a ``simple lattice") of $\R^2$ is given by $L=\Z u\oplus\Z v$ where $(u,v)$ is a basis of $\R^2$. By Engel's theorem (see \cite{Engel}), we can choose $u$ and $v$ so that $\|u\|\leq \|v\|$ and $\displaystyle (\widehat{u,v})\in\left[\frac{\pi}{3},\frac{\pi}{2}\right]$ in order to obtain the unicity of the lattice, up to a rotation. We note $|L|=\|u\land v\|=\|u\|\|v\|\left|\sin(\widehat{u,v})\right|$ the area of $L$ which is in fact the area of the lattice primitive cell and $L^*:=L\backslash \{0\}$. The positive definite quadratic form associated with the Bravais lattice $L$ is, for $(m,n)\in\Z^2$, 
$$
Q_L(m,n)=\|mu+nv\|^2=\|u\|^2m^2+\|v\|^2n^2+2\|u\|\|v\|\cos(\widehat{u,v})mn.
$$
For a positive definite quadratic form $q(m,n)=am^2+bmn+cn^2$, we define its discriminant $D=4ac-b^2\geq 0$. Hence for $Q_L$, we obtain : 
$$
 D=4\|u\|^2\|v\|^2-4\|u\|^2\|v\|^2\cos^2(\widehat{u,v})=4\|u\|^2\|v\|^2\sin^2(\widehat{u,v})=4|L|^2.
$$
In this paper, the term ``lattice" will mean a ``Bravais lattice", and we define, for $s>2$, the Epstein zeta function of the lattice $L$ by 
$$
\zeta_L(s):=\sum_{x\in L^*}\frac{1}{\|x\|^s}=\sum_{(m,n)\neq (0,0)}\frac{1}{Q_L(m,n)^{s/2}}.
$$
Let $\displaystyle\Lambda_A=\sqrt{\frac{2A}{\sqrt{3}}}\left[\Z (1,0) \oplus \Z (1/2,\sqrt{3}/2)\right]$ be the triangular lattice of area $A$, also called the hexagonal lattice. Its length is the norm of its vector $u$, i.e. the minimum distance strictly positive of $\Lambda_A$, $\|u\|=\sqrt{2A/\sqrt{3}}$. We notice, for any $s>2$, that
\begin{equation}\label{zetascale}
 \zeta_{\Lambda_A}(s)=\frac{\zeta_{\Lambda_1}(s)}{A^{s/2}}
 \end{equation}
 and this relation of scaling is true for any lattice $L$ of area $A$.\\
We recall the result of Montgomery about theta functions :
\begin{thm}\label{Mgt} \textnormal{(Montgomery, \cite{Mont})} For any real number $\alpha>0$ and a Bravais lattice $L$, let
$$
\theta_L(\alpha):=\Theta_L(i\alpha)=\sum_{m,n\in\Z}e^{-2\pi\alpha Q_L(m,n)},
$$
where $\Theta_L$ is the Jacobi theta function of the lattice $L$ defined for $Im(z)>0$. Then, for any $\alpha>0$, $\Lambda_A$ is the unique minimizer of $L\to \theta_L(\alpha)$ among lattices of area $A$, up to rotation.
\end{thm}
\begin{remark} The same kind of results were obtained by Nonnenmacher and Voros in \cite{NonnenVoros}. The previous theorem implies that the triangular lattice is the unique minimizer, up to rotation, of $L\mapsto\zeta_L(s)$ among lattices with density fixed for any $s>2$ which is also proved by Rankin (in \cite{Rankin}).
\end{remark} 
\noindent We consider the classical Lennard-Jones potential 
$$V_{LJ}(r)=\frac{1}{r^{12}}-\frac{2}{r^6}
$$
whose minimum is obtained at $r=1$, and for $L=\Z u\oplus\Z v$ a Bravais lattice of $\R^2$, we let 
$$
E_{LJ}(L):=\sum_{x\in L^*}V_{LJ}(\|x\|)=\zeta_L(12)-2\zeta_L(6)
$$
be  the Lennard-Jones energy of lattice $L$. By \eqref{zetascale} this energy among lattices of area $A$ can be viewed as energy $L\mapsto E_{LJ}(\sqrt{A}L)$ over lattices of area $1$ and we parametrize $L$ with its length $\|u\|$ and $\|v\|$ by
$$
Q_L(m,n)=\|u\|^2m^2+\|v\|^2n^2+2mn\sqrt{\|u\|^2\|v\|^2-1}.
$$
It follows that we can write Lennard-Jones energy among lattices of area $A$ as
\begin{equation} \label{parametrize}
(\|u\|,\|v\|)\mapsto \sum_{(m,n)\neq (0,0)}V_{LJ}\left(\sqrt{A}\sqrt{\|u\|^2m^2+\|v\|^2n^2+2mn\sqrt{\|u\|^2\|v\|^2-1}} \right).
\end{equation}
\begin{center}
\includegraphics[width=8cm,height=60mm]{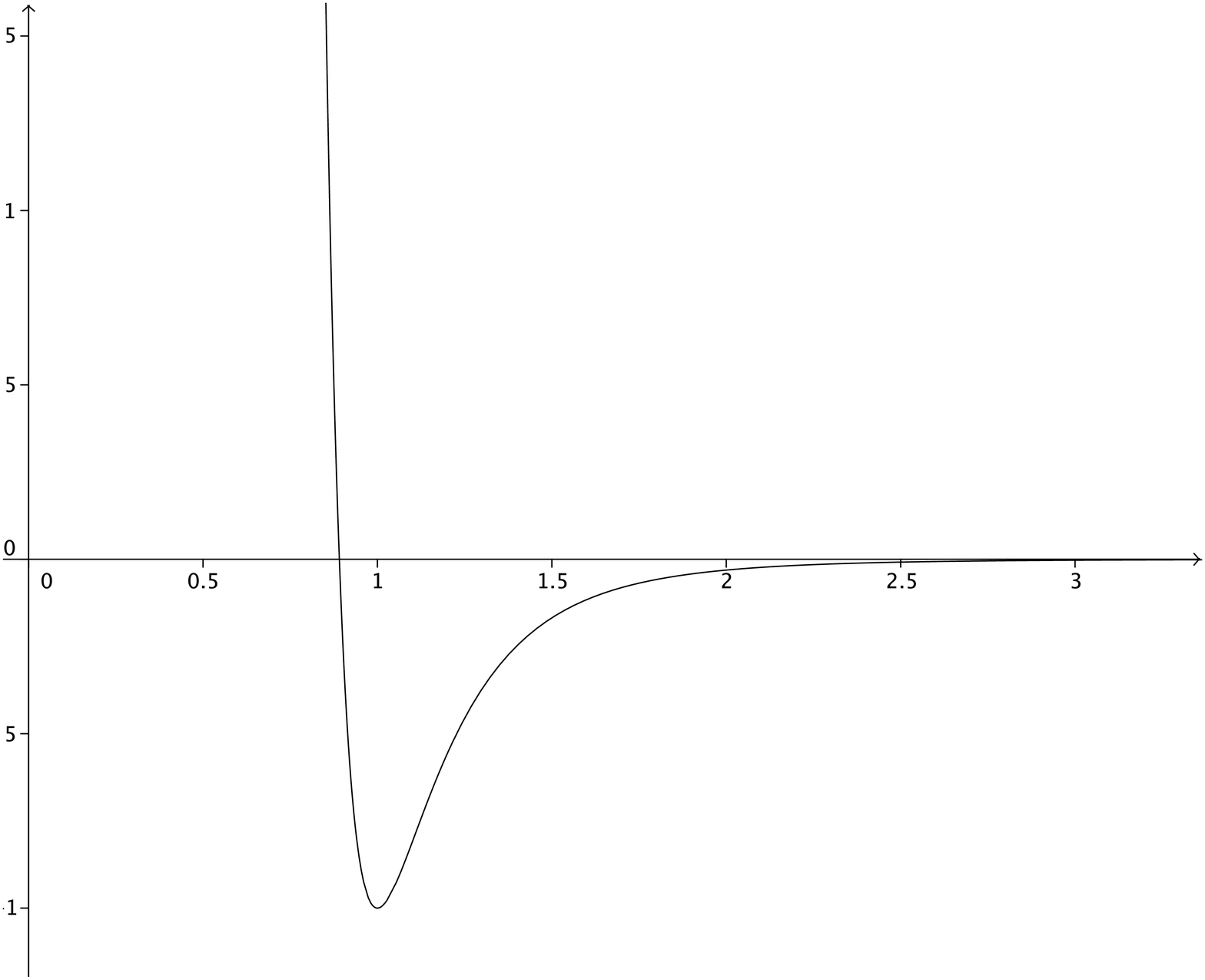}\\
\textbf{Fig. 1:} \textit{Graph of the Lennard-Jones potential $V_{LJ}$}
\end{center}
The aim of this paper is to study the following two minimization problems, up to rotation :
\begin{align*}
&(P_A):\text{ Find the minimizer of $E_{LJ}$ among lattices $L$ with fixed $|L|=A$};\\
&(P):\text{ Find the minimizer of $E_{LJ}$ among lattices}.
\end{align*}

\begin{prop} The minimum of $E_{LJ}$ among lattices is achieved.
\end{prop}
\begin{proof} We parametrize a lattice $L$ by $x=\|u\|$, $y=\|v\|$ and $\theta=(\widehat{u,v})$, therefore 
\begin{align*}
&f(x,y,\theta):=E_{LJ}(L)\\
&=\sum_{(m,n)\neq(0,0)}\left(\frac{1}{(x^2m^2+y^2n^2+2xymn\cos\theta)^6}-\frac{2}{(x^2m^2+y^2n^2+2xymn\cos\theta)^3}  \right).
\end{align*}
\textit{First case : minimization without fixed area.}
If $L$ is the solution of $(P)$ then $x$ and $y$ cannot be too small, otherwise the energy is too large and a proof of a lower bound for $x$ is given in Section 4. Moreover $y\leq 1$ because if $y>1$ then a contraction of the line $\R v$ gives smaller energy. Therefore we have $x,y\in [m,M]$ and $\theta\in [\pi/3,\pi/2]$. The function $(x,y,\theta)\mapsto f(x,y,\theta)$ is continuous on $[m,M]\times[m,M]\times [\pi/3,\pi/2]$ hence its minimum is achieved.\\
\textit{Second case : minimization with fixed area.} We can parametrize $L$ with only two variables $x$ and $y$ -- as in \eqref{parametrize} -- such that when $x\to 0$ then $y\to +\infty$. As $L$ should be a Bravais lattice, it is clear that the minimum of $f$ is achieved.
\end{proof}

\section{Minimization among lattices with fixed area}

\subsection{A sufficient condition for the minimality of $E_{LJ}$ : Montgomery's method}
Our idea is to write $E_{LJ}$ in terms of $\theta_L$ and to use Theorem \ref{Mgt} in order to find a sufficient condition for the minimality of the triangular lattice among Bravais lattices with a fixed area. 
\begin{thm}\label{ThM}
If  $\displaystyle A^3\leq \frac{\pi^3}{120}$, then $\Lambda_A$ is the unique solution of $(P_A)$.\\
\end{thm}

\begin{proof} As it is explained in \cite{Mont} or \cite{Terras}, we can write the Epstein zeta function in terms of a theta function. Indeed, we have the following identity, where the discriminant of $Q_L$ is $D=1$ :
\begin{equation} \label{Riemann}
\text{for } \textnormal{Re}(s)>1, \zeta_L(2s)\Gamma(s)(2\pi)^{-s}=\frac{1}{s-1}-\frac{1}{s}+\int_1^\infty(\theta_L(\alpha)-1)(\alpha^s+\alpha^{1-s})\frac{d\alpha}{\alpha}.
\end{equation}
Thus, for $|L|=A$, we write $E_{LJ}(L)=\zeta_L(12)-2\zeta_L(6)$ as an integral $\displaystyle \int_1^{+\infty}g_A(\alpha)\left(\theta_L\left(\frac{\alpha}{2A}\right)-1\right)\frac{d\alpha}{\alpha}$, up to a constant independent of $L$ and we find $A$ so that $g_A(\alpha)\geq 0$ for any $\alpha\geq1$. As $\Lambda_A$ is the unique minimizer of $\theta_L(\alpha)$ for any $\alpha>0$, we have for any $L$ such that $|L|=A$ :
$$
E_{LJ}(L)-E_{LJ}(\Lambda_A)=\int_1^{+\infty}\left(\theta_L\left(\frac{\alpha}{2A}\right)-\theta_{\Lambda_A}\left(\frac{\alpha}{2A}\right)\right)g_A(\alpha)\frac{d\alpha}{\alpha}\geq 0
$$
and $\Lambda_A$ is the unique solution of $(P_A)$.\\
In fact \eqref{Riemann} it is the classic ``Riemann's trick" and here we will briefly recall its proof : as 
$$ \Gamma(s)(2\pi)^{-s}Q_L(m,n)^{-s}=\int_0^\infty t^{s-1}e^{-t}(2\pi)^{-s}Q_L(m,n)^{-s}dt
$$
for $\text{Re}(s)>1$, and by putting $t=2\pi Q_L(m,n)y$, we obtain 
$$
\Gamma(s)(2\pi)^{-s}Q_L(m,n)^{-s}=\int_0^\infty e^{-2\pi y Q_L(m,n)}y^{s-1}dy.
$$
Summing over $(m,n)\neq (0,0)$ and using the identity $\theta_L(1/\alpha)=\alpha\theta_L(\alpha)$ for any $\alpha>0$, proved by Montgomery in \cite{Mont}, we obtain 
\begin{align*}
\Gamma(s)(2\pi)^{-s}\zeta_L(2s)&=\int_0^\infty(\theta_L(y)-1)y^{s-1}dy =\int_0^1(\theta_L(y)-1)y^{s-1}dy+\int_1^\infty(\theta_L(y)-1)y^{s-1}dy\\
&=\int_1^\infty(\theta_L(1/\alpha)-1)\alpha^{-1-s}d\alpha+\int_1^\infty(\theta_L(\alpha)-1)\alpha^{s-1}d\alpha\\
&=\int_1^\infty(\alpha\theta_L(\alpha)-1)\alpha^{-1-s}d\alpha+\int_1^\infty(\theta_L(\alpha)-1)\alpha^{s-1}d\alpha\\
&=\int_1^\infty\theta_L(\alpha)\alpha^{-s}d\alpha-\int_1^\infty\alpha^{-1-s}d\alpha+\int_1^\infty(\theta_L(\alpha)-1)\alpha^{s-1}d\alpha\\
&=\int_1^\infty(\theta_L(\alpha)-1)\alpha^{-s}d\alpha+\int_1^\infty(\theta_L(\alpha)-1)\alpha^{s-1}d\alpha+ \int_1^\infty \alpha^{-s}d\alpha-\int_1^\infty\alpha^{-1-s}d\alpha\\
&=\int_1^\infty(\theta_L(\alpha)-1)\alpha^{-s}d\alpha+\int_1^\infty(\theta_L(\alpha)-1)\alpha^{s-1}d\alpha+\frac{1}{s-1}-\frac{1}{s}\\
&=\int_1^\infty(\theta_L(\alpha)-1)(\alpha^s+\alpha^{1-s})\frac{d\alpha}{\alpha}+\frac{1}{s-1}-\frac{1}{s}.
\end{align*}
Now if $|L|=A$, by the equality $D=(2A)^2$ there are two identities :
$$
(2\pi)^{-6}(2A)^6\Gamma(6)\zeta_L(12)=\frac{1}{5}-\frac{1}{6}+\int_1^{+\infty}\left(\theta_L\left(\frac{\alpha}{2A}\right)-1\right)(\alpha^6+\alpha^{1-6})\frac{d\alpha}{\alpha}
$$

$$
(2\pi)^{-3}(2A)^3\Gamma(3)\zeta_L(6)=\frac{1}{2}-\frac{1}{3}+\int_1^{+\infty}\left(\theta_L\left(\frac{\alpha}{2A}\right)-1\right)(\alpha^3+\alpha^{1-3})\frac{d\alpha}{\alpha}
$$
and we find
$$
\zeta_L(12)=\frac{(2\pi)^6}{30(2A)^65!}+\int_1^{+\infty}\left(\theta_L\left(\frac{\alpha}{2A}\right)-1\right)\frac{(2\pi)^6}{(2A)^65!}(\alpha^6+\alpha^{-5})\frac{d\alpha}{\alpha}
$$
$$
\zeta_L(6)=\frac{(2\pi)^3}{6(2A)^32!}+\int_1^{+\infty}\left(\theta_L\left(\frac{\alpha}{2A}\right)-1\right)\frac{(2\pi)^3}{(2A)^32!}(\alpha^3+\alpha^{-2})\frac{d\alpha}{\alpha}.
$$
Therefore, for any $L$ of area $A$,
$$
E_{LJ}(L)=C_{A}+\frac{\pi^3}{A^3}\int_1^{+\infty}\left(\theta_L\left(\frac{\alpha}{2A}\right)-1\right)g_{A}(\alpha)\frac{d\alpha}{\alpha}
$$
where $\displaystyle g_{A}(\alpha):=\frac{\pi^3}{A^35!}(\alpha^6+\alpha^{-5})-(\alpha^3+\alpha^{-2})$, and $C_{A}$ is a constant depending on $A$ but independent of $L$. Now we want to prove that if $\pi^3\geq 120 A^3$ then $g_A(\alpha)\geq 0$ for any $\alpha\geq 1$. First, we remark that 
$$ 
g_{A}(1)\geq 0 \iff \frac{\pi^3}{A^35!}-1\geq 0 \iff  \pi^3\geq 120 A^3.
$$
Secondly, we compute $\displaystyle g_{A}'(\alpha)=\frac{\pi^3}{A^35!}(6\alpha^5-5\alpha^{-6})-(3\alpha^2-2\alpha^{-3})$,
and if $\displaystyle \pi^3\geq 120 A^3$ then 
$$
g_{A}'(1)=\frac{\pi^3}{A^35!}-1\geq 0.
$$
Finally, we compute $\displaystyle g_{A}''(\alpha)= \frac{\pi^3}{A^35!}(30\alpha^4+30\alpha^{-7})-(6\alpha+6\alpha^{-4})$.
As $\displaystyle \frac{\pi^3}{A^35!}\geq1$ and $\alpha\geq 1$, 
$$
 \frac{\pi^3}{A^35!}(30\alpha^4+30\alpha^{-7})-(6\alpha+6\alpha^{-4})\geq 30\alpha^4+30\alpha^{-7}-6\alpha-6\alpha^{-4}\geq 24\alpha +30\alpha^{-7}-6\alpha^{-4}\geq 0.
 $$
Thus, we have shown that, for any $A$ so that $\pi^3\geq 120 A^3$, $g_{A}''(\alpha)\geq 0$ for any $\alpha\geq 1$, $g_{A}'(1)\geq 0$ and $g_{A}(1)\geq 0$. Hence $g_{A}(\alpha)\geq 0$ for any $\alpha\geq 1$ if $\pi^3\geq 120 A^3$.
\end{proof}
\begin{remark} We have $\displaystyle \left( \frac{\pi^3}{120} \right)^{1/3}\approx 0.63693$, hence for $A\leq 0.63692$, $\Lambda_A$ is the unique solution of $(P_A)$.
\end{remark}
\begin{remark} We prove below (see Proposition \ref{UpB})  that when $A$ is sufficiently large then $\Lambda_A$ is no longer a solution of $(P_A)$. However, our bound $\pi^3\geq 120 A^3$ is likely not to be optimal. If it were, by the Proposition \ref{tri} and its remark, then the triangular lattice is not the solution to $(P)$.
\end{remark}

\noindent This result explains that the behaviour of the potential is important for the interaction between the first neighbours because in this case the reverse power part $\displaystyle r^{-12}$ is the strongest interaction. This method can be adapted to any potential of the form $\displaystyle V(r)=\frac{K_1}{r^n}-\frac{K_2}{r^p}$ with $n>p>2$ to obtain similar results in two dimensions.\\ 
\begin{remark} \textbf{The three-dimensional case} is an open problem. Indeed, there is no result related to the minimization of theta and Epstein functions among Bravais lattices of $\R^3$ with fixed volume. Sarnak and Str\"ombergsson recalled in \cite{SarStromb} that Ennola had shown in \cite{Ennola} the local minimality of the face centred cubic lattice for $\zeta_L(s)$ and for any $s>0$. They also prove that the face centred cubic lattice cannot be the minimizer of $\zeta_L(s)$ for all $s>0$. Hence the problem of minimization of Lennard-Jones energy among lattices of $\R^3$, and of course in higher dimensions, seems to be very difficult.
\end{remark}

\subsection{A necessary condition for the minimality of the triangular lattice for $E_{LJ}$}

\begin{prop}\label{UpB}
$\Lambda_A$ is a solution of $(P_A)$ if and only if $\displaystyle A\leq \inf_{|L|=1 \atop L\neq \Lambda_1}\left( \frac{\zeta_{L}(12)-\zeta_{\Lambda_1}(12)}{2(\zeta_{L}(6)-\zeta_{\Lambda_1}(6))} \right)^{1/3}$.\\
Hence if $A$ is sufficiently large, $\Lambda_A$ is not a solution of $(P_A)$.
\end{prop}
\begin{proof} We have the following equivalences
\begin{align*}
&E_{LJ}(\Lambda_A)\leq E_{LJ}(L) \text{ for any $L$ such that } |L|=A\\
&\iff \zeta_{\Lambda_A}(12)-2\zeta_{\Lambda_A}(6)\leq \zeta_L(12)-2\zeta_L(6) \text{ for any $L$ such that } |L|=A\\
&\iff 2(\zeta_L(6)-\zeta_{\Lambda_A}(6))\leq \zeta_{L}(12)-\zeta_{\Lambda_A}(12) \text{ for any $L$ such that } |L|=A\\
&\iff \frac{2(\zeta_L(6)-\zeta_{\Lambda_1}(6))}{A^3}\leq \frac{\zeta_{L}(12)-\zeta_{\Lambda_1}(12)}{A^6}  \text{ for any $L$ such that } |L|=1
\end{align*}
by the scaling property \eqref{zetascale}. We recall that $\zeta_L(6)>\zeta_{\Lambda_1}(6)$ for any $L$ of area $A$ so that $L\neq \Lambda_1$, as a consequence of Theorem \ref{Mgt} and the Riemann's trick \eqref{Riemann}. Then we obtain
\begin{align*}
&E_{LJ}(\Lambda_A)\leq E_{LJ}(L) \text{ for any $L$ such that } |L|=A\\
&\iff A\leq \inf_{|L|=1 \atop L\neq \Lambda_1}\left( \frac{\zeta_{L}(12)-\zeta_{\Lambda_1}(12)}{2(\zeta_{L}(6)-\zeta_{\Lambda_1}(6))} \right)^{1/3}.
\end{align*}
\end{proof}
\noindent It is difficult to study the minimum of function $\displaystyle L\mapsto\left( \frac{\zeta_{L}(12)-\zeta_{\Lambda_1}(12)}{2(\zeta_{L}(6)-\zeta_{\Lambda_1}(6))}\right)^{1/3}$ among lattices $L\neq \Lambda_1$ such that $|L|=1$. However, we can numerically look for a lower bound. This function can be parametrized with two variables -- here the lengths $\|u\|$ and $\|v\|$ of the lattice $L$ as in \eqref{parametrize} -- and we can plot the level sets of it. We notice that the large differences between the values of the function only give a domain where the function is minimum.
\begin{center}
\includegraphics[width=8cm,height=70mm]{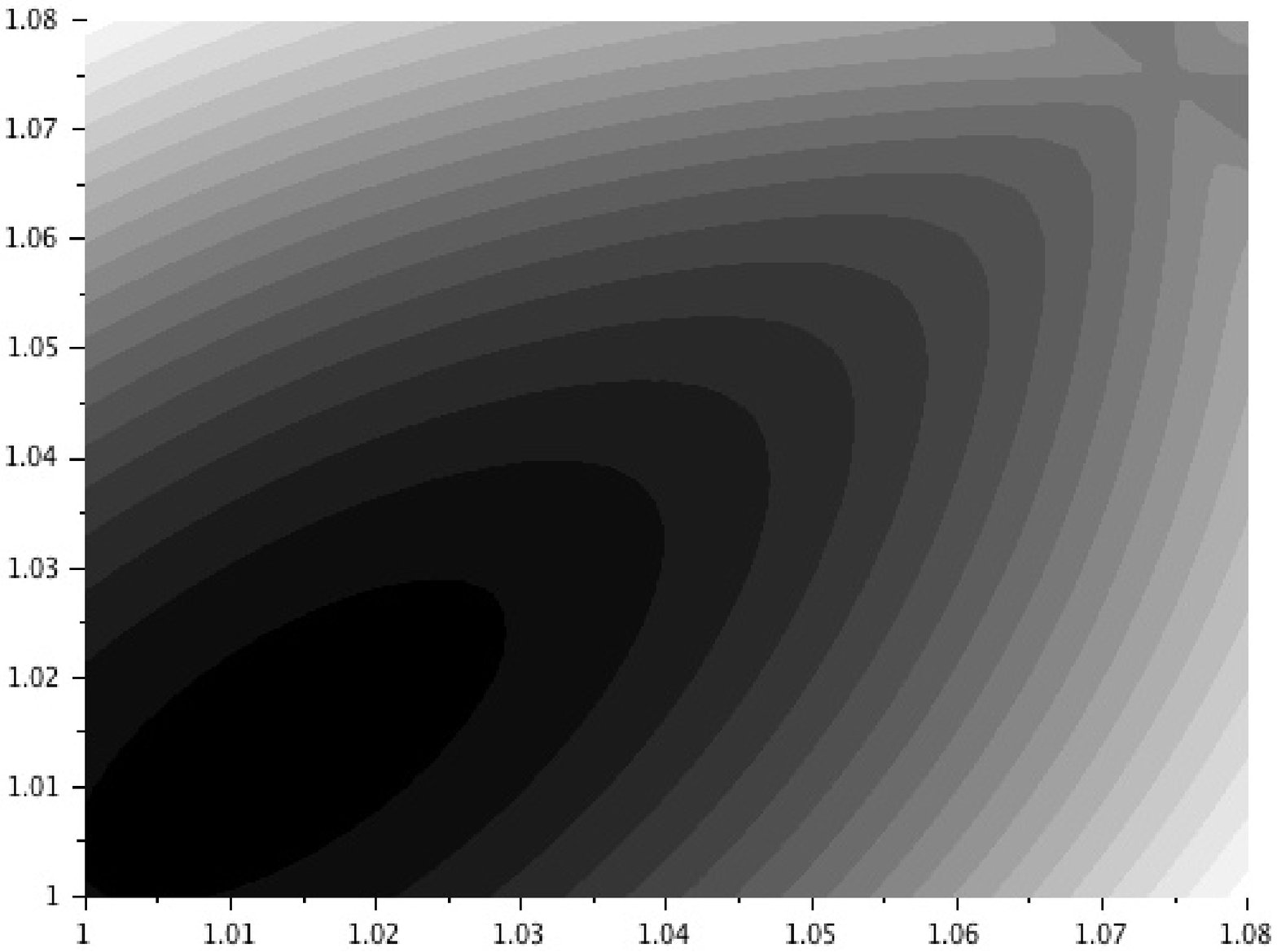} \includegraphics[width=8cm,height=70mm]{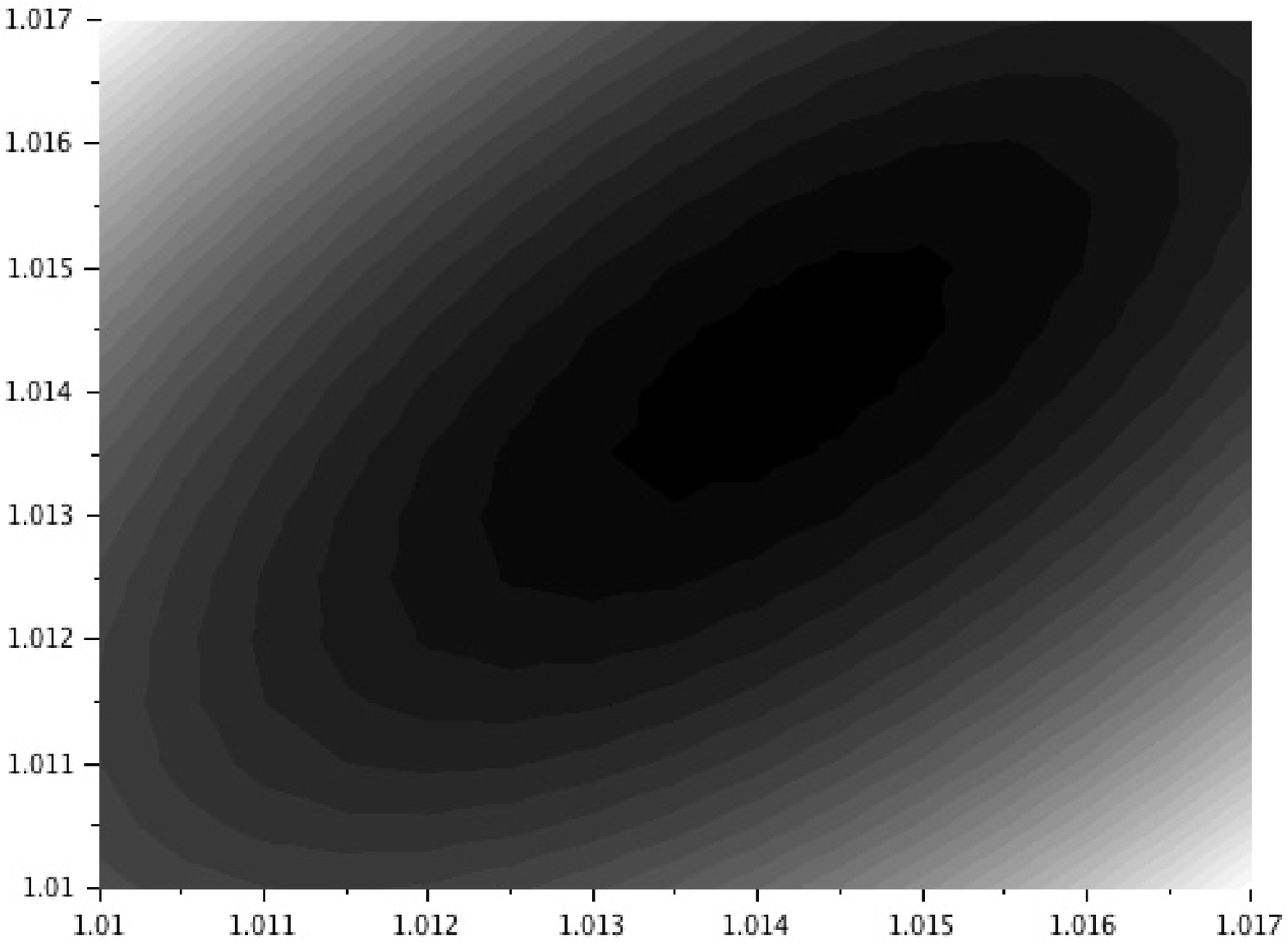}
\textbf{Fig. 2 :} \textit{Level sets of $\displaystyle (\|u\|,\|v\|)\mapsto \left(\frac{\zeta_{L}(12)-\zeta_{\Lambda_1}(12)}{2(\zeta_{L}(6)-\zeta_{\Lambda_1}(6))}\right)^{1/3}$}\\
(black = minimum, white = maximum)
\end{center}
Indeed, its minimum seems to be around lattice $L$ of area $1$ such that $\|u\|=\|v\|=1.014$ and for this one, we have $\displaystyle\left( \frac{\zeta_{L}(12)-\zeta_{\Lambda_1}(12)}{2(\zeta_{L}(6)-\zeta_{\Lambda_1}(6))}\right)^{1/3}\approx 1.1378475$, hence numerically the minimum of this function is between $1.13$ and $1.14$.\\ \\
Actually Fig. 3 gives the Lennard-Jones energy -- viewed as a function of two variables $\|u\|$ and $\|v\|$ over the lattices of area one (see \eqref{parametrize}) -- for $(\|u\|,\|v\|)\in [1,1.08]^2$. The triangular lattice $\Lambda_1$ corresponds to the point $\left(\sqrt{2/\sqrt{3}},\sqrt{2/\sqrt{3}}\right)\approx (1.075,1.075)$ and the square lattice $\Z^2$ corresponds to the point $(1,1)$. In fact it is clear that the point associated with the triangular lattice is a critical point of this energy, because the triangular lattice is the unique minimizer of Epstein zeta function among lattices of area $A$. Moreover we can prove that the square lattice is also a critical point, by using an other parametrization as $(\|u\|,\theta)$. We numerically obtain :
\begin{itemize}
\item For $A=1$, $\Lambda_1$ seems to be its minimizer and $\Z^2$ is a local maximizer.
\item For $A=1.13$, $\Lambda_1$ seems to be its minimizer but $\Z^2$ seems to be not a local maximizer.
\item For $A=1.14$, $\Z^2$ seems to be its minimizer because we estimate $E_{LJ}(\sqrt{1.14}\Lambda_1)\approx -4.435$ is larger than $E_{LJ}(\sqrt{1.14}\Z^2)\approx -4.437$ 
\item For $A=1.16$, $\Z^2$ seems to be its minimizer.
\item For $A=1.2$, $\Z^2$ seems to be its minimizer and $\Lambda_1$ is a local maximizer.
\item For $A=2$ (and more), $\Z^2$ seems to be its minimizer and $\Lambda_1$ is a local maximizer.
\end{itemize}
\noindent Hence, we can write the following conjecture based on our numerical study of $L\mapsto E_{LJ}(\sqrt{A}L)$ among all lattices with area $1$ :\\ \\
\textbf{Conjecture :} \textit{If $A$ is sufficiently large, the square lattice is the unique solution of $(P_A)$.}\\
\begin{tabular}{cc}
\includegraphics[width=8cm,height=70mm]{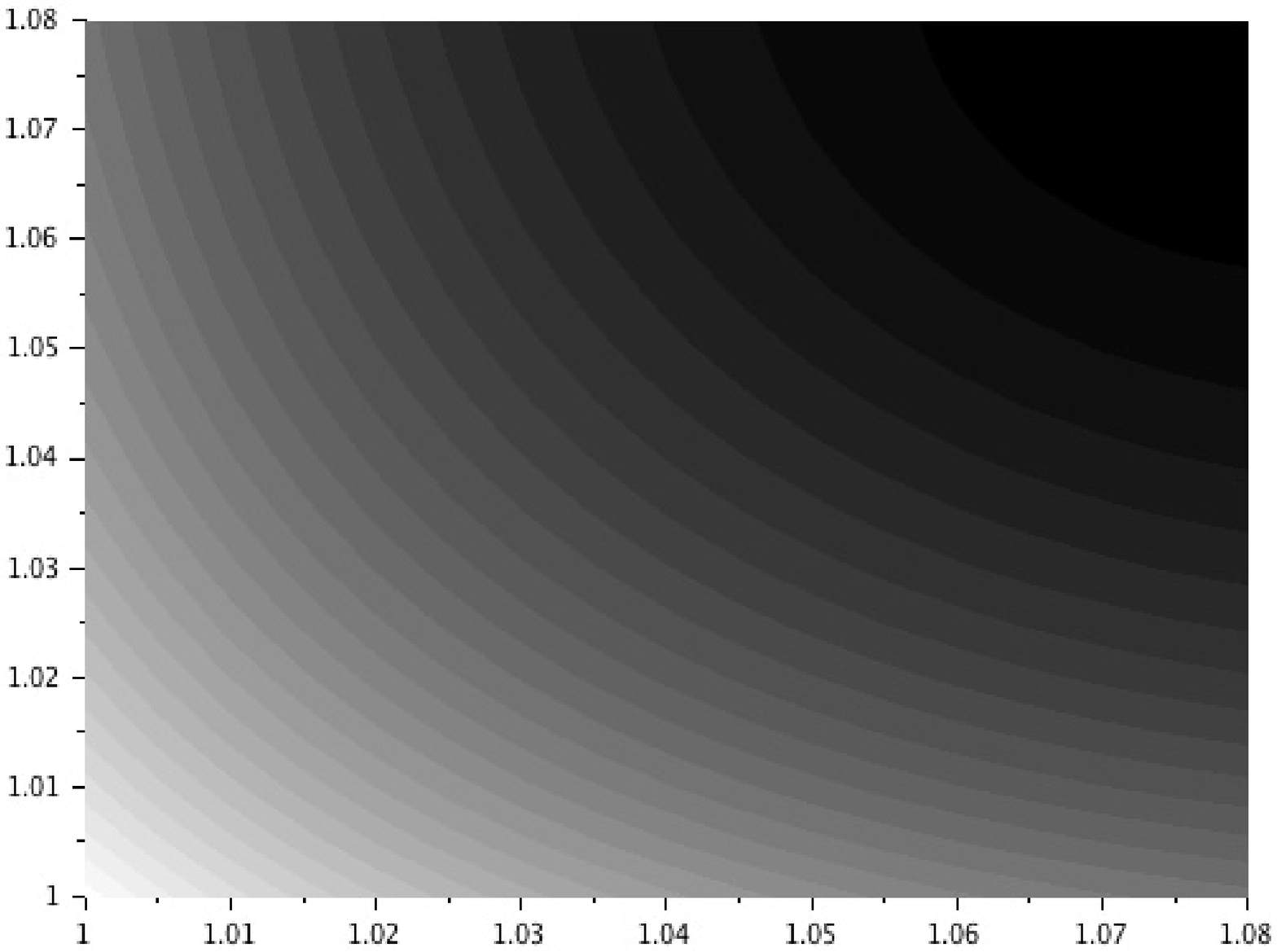} & \includegraphics[width=8cm,height=70mm]{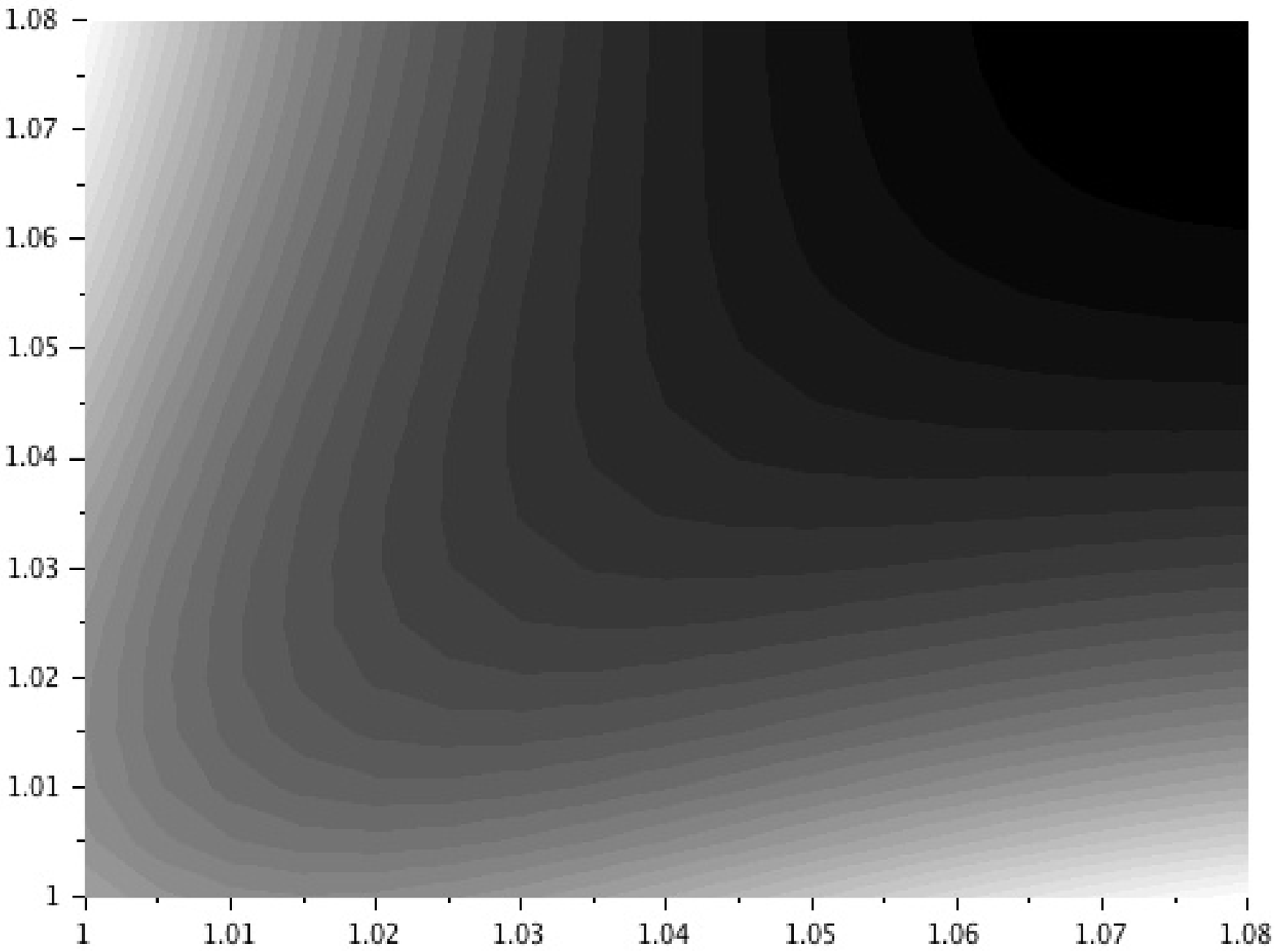} \\
A=1 & A=1.13
\end{tabular}
 \begin{tabular}{cc}
\includegraphics[width=8cm,height=70mm]{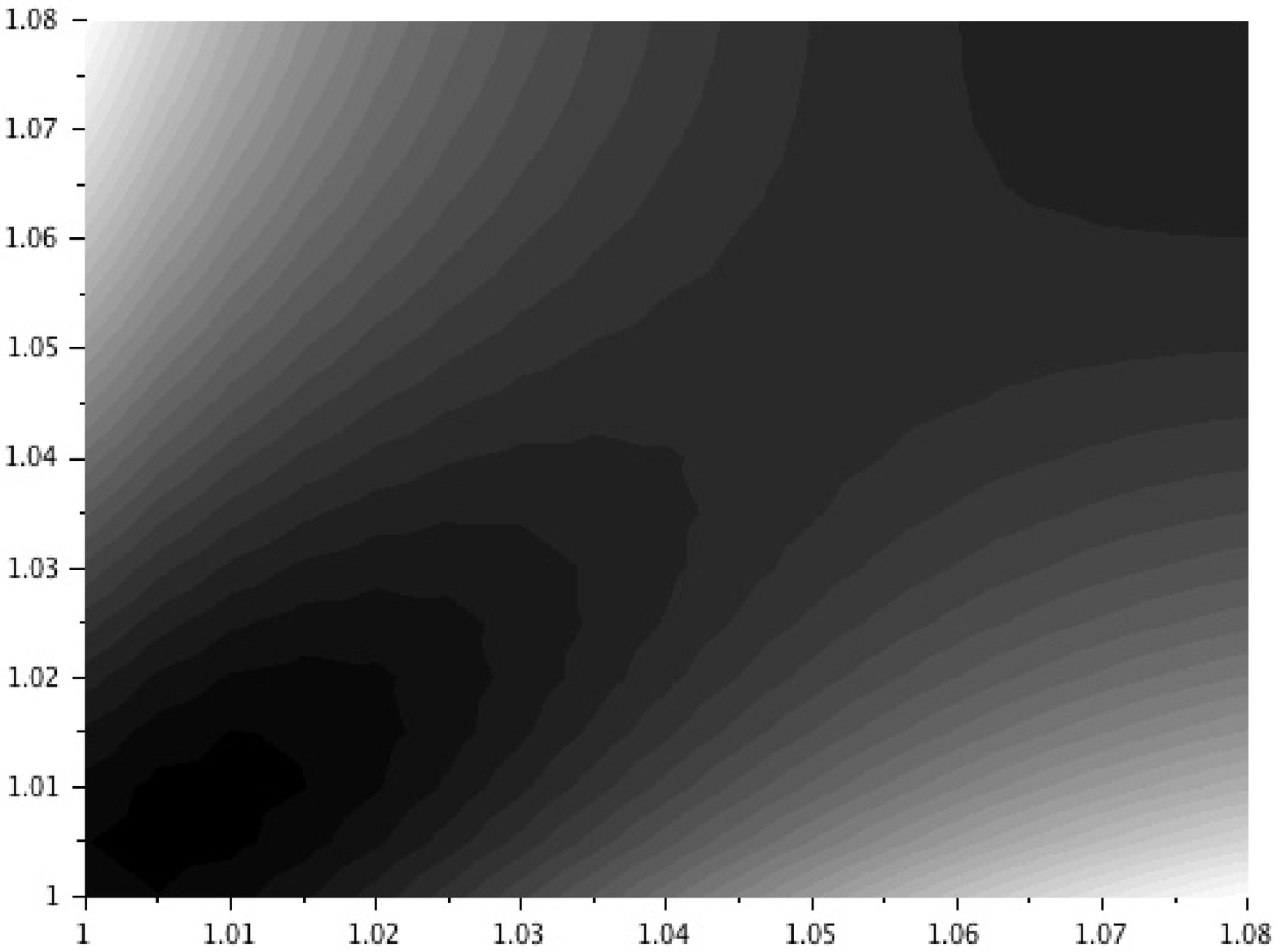} & \includegraphics[width=8cm,height=70mm]{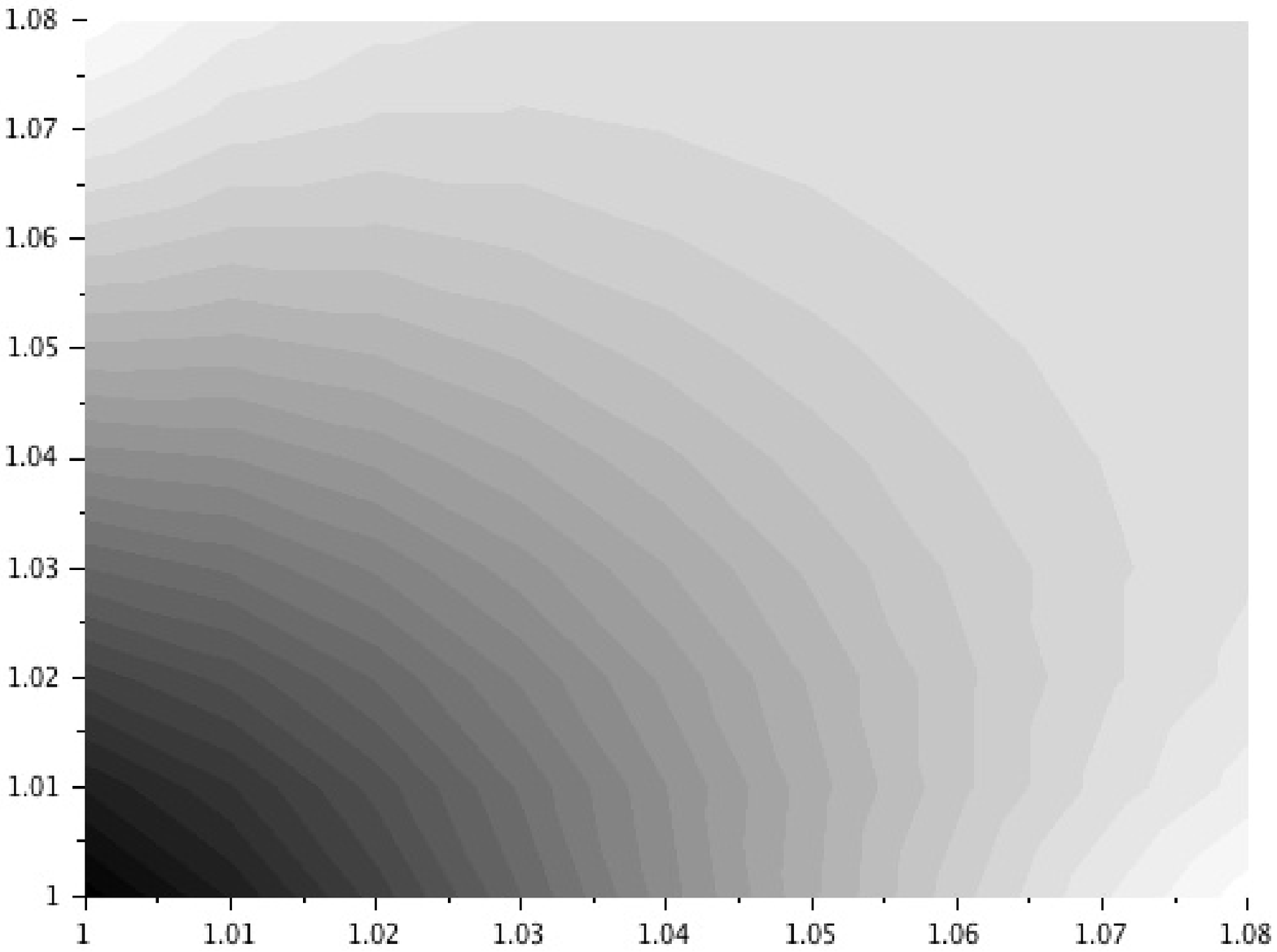} \\
A=1.14 & A=1.16
\end{tabular}
\begin{tabular}{cc}
\includegraphics[width=8cm,height=70mm]{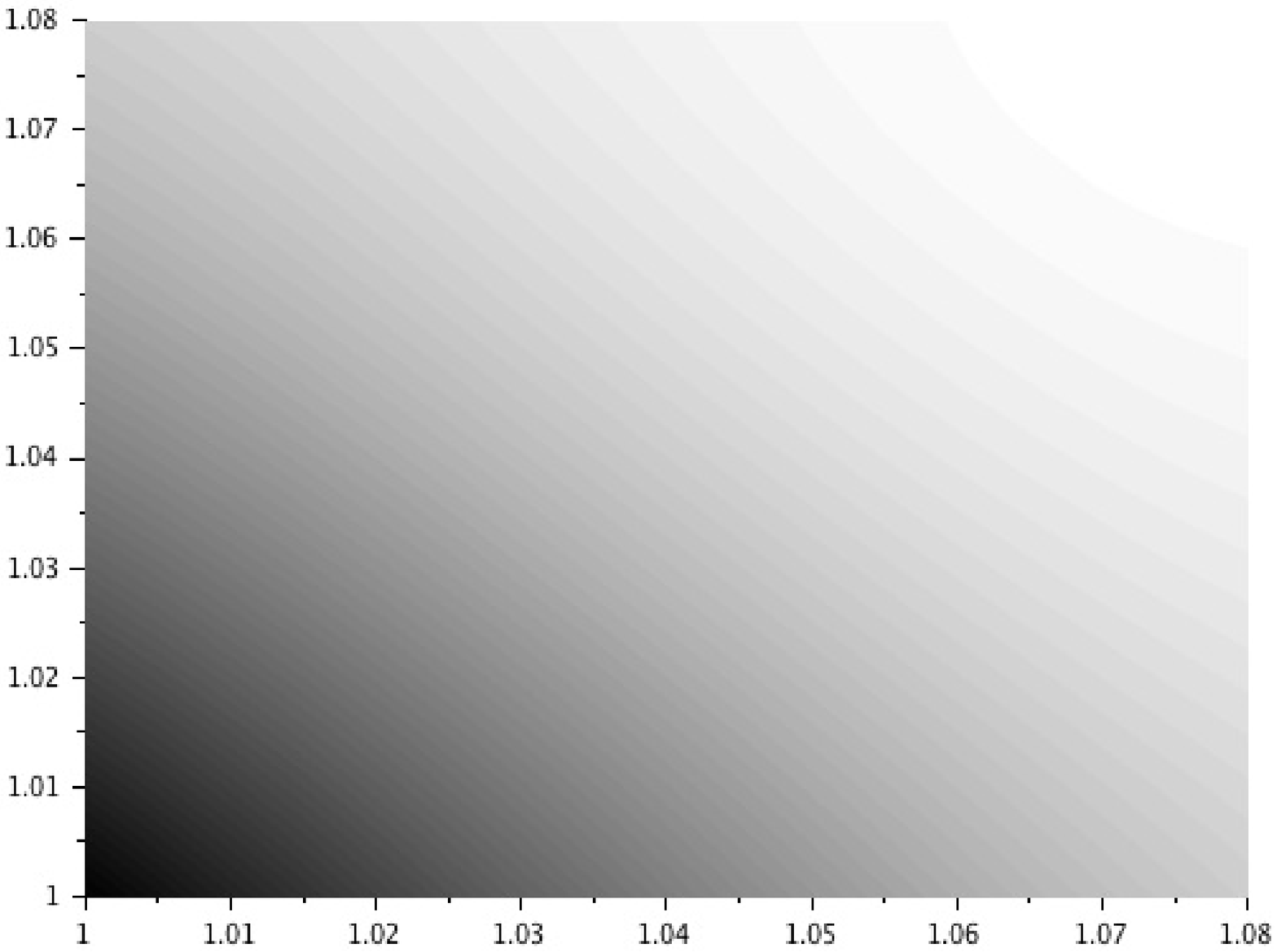} & \includegraphics[width=8cm,height=70mm]{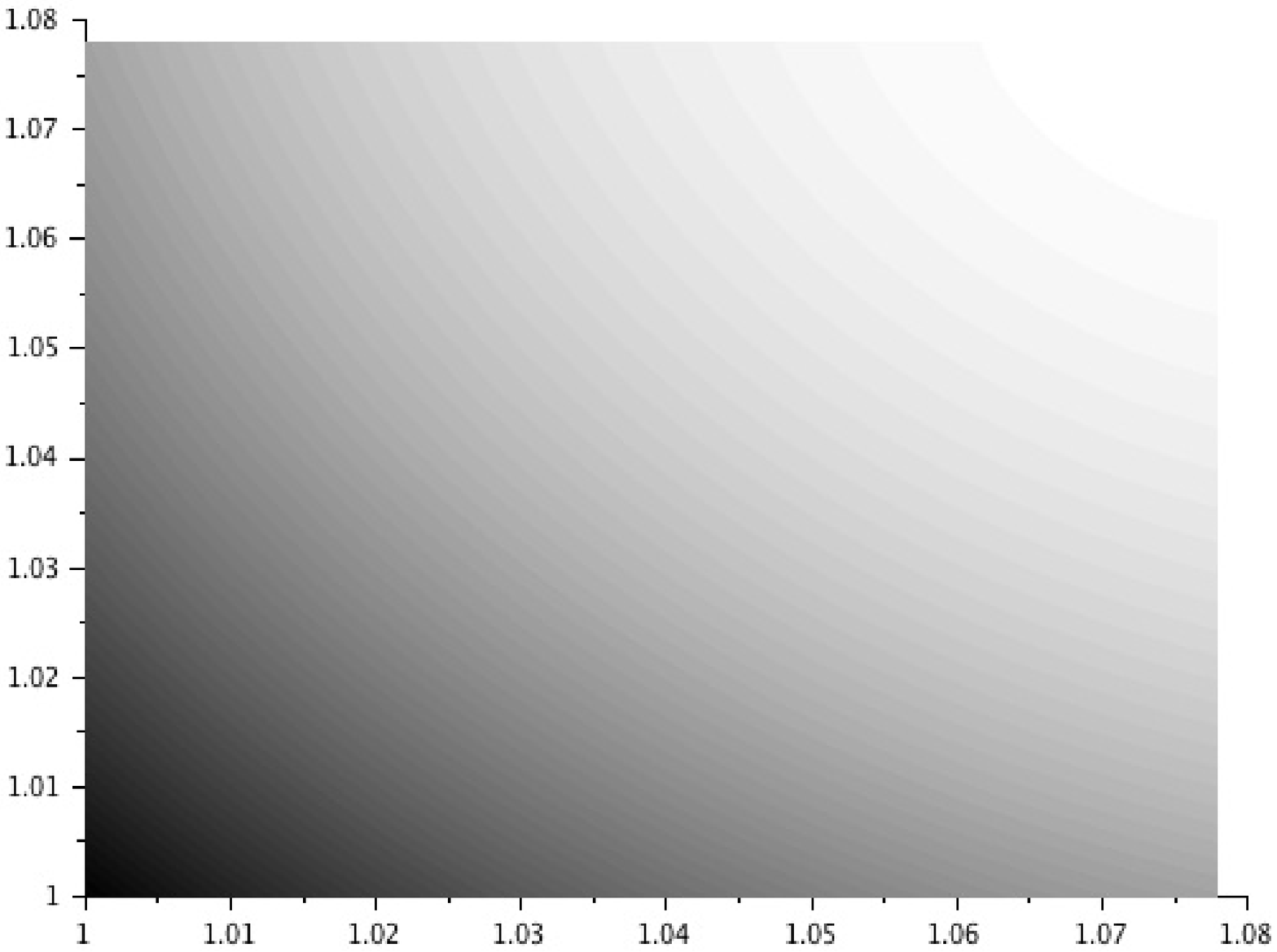}\\
A=1.2 &  A=2\\
\end{tabular}
\begin{center}
\textbf{Fig. 3 :} \textit{Level sets of $(\|u\|,\|v\|)\mapsto E_{LJ}(\sqrt{A}L)$ for some interesting values of $A$}\\
(black = minimum , white = maximum)
\end{center}
\section{Global minimization of $E_{LJ}$ among lattices}
Now we study the problem $(P)$. We give high properties for the global minimizer among lattices and some indications of its shape.
\subsection{Characterization of the global minimizer}
\begin{prop}\label{globmin}
If $L_0=\Z u\oplus \Z v$ is a solution of $(P)$ then \\
$i)$ $\displaystyle E_{LJ}(L_0)=-\zeta_{L_0}(6)=-\zeta_{L_0}(12)< 0$,\\
$ii)$ $\|u\|<1$ and $\|v\|\leq 1$,\\
$iii)$ $\zeta_{L_0}(6)=\max\{ \zeta_L(6) ; L \text{ such that } \zeta_L(12)\leq \zeta_L(6) \}$.
\end{prop}
\begin{proof} 
$i)$ We consider the function $\displaystyle f(r)=E_{LJ}(rL_0)=r^{-12}\zeta_{L_0}(12)-2r^{-6}\zeta_{L_0}(6)$. As $L_0$ is a global minimizer of $E_{LJ}$, $r=1$ is the critical point of $f$ and $f'(r)=-12r^{-13}\zeta_{L_0}(12)+12r^{-7}\zeta_{L_0}(6)$, hence
$$
f'(1)=0 \iff \zeta_{L_0}(12)=\zeta_{L_0}(6)
$$
and $\displaystyle E_{LJ}(L_0)=\zeta_{L_0}(12)-2\zeta_{L_0}(6)=-\zeta_{L_0}(6)=-\zeta_{L_0}(12)$.\\
\\
$ii)$ As $\zeta_{L_0}(12)=\zeta_{L_0}(6)$, it is clear that $\|u\|< 1$ because if $r>1$ then $r^{-12}<r^{-6}$. If $\|v\|>1$, a little contraction of $\R v$ yields a new lattice $L_1$ such that $E_{LJ}(L_1)<E_{LJ}(L_0)$ because some of the distances of the lattice decrease while $\|u\|$ is constant, therefore the energy decreases.\\
$iii)$ $-\zeta_{L_0}(6)=E_{LJ}(L_0)\leq E_{LJ}(L) \iff \zeta_L(6)-\zeta_{L_0}(6)\leq \zeta_L(12)-\zeta_L(6)$ and if $L$ is a lattice such that $\zeta_L(12)\leq \zeta_L(6)$, we get $\zeta_{L}(6)\leq \zeta_{L_0}(6)$.
\end{proof}
\begin{corollary} The triangular lattice of length $1$ cannot be the solution of $(P)$ though the minimum of the potential $V_{LJ}$ is achieved for $r=1$.
\end{corollary}
\begin{prop} \label{tri} The minimizer of $E_{LJ}$ among triangular lattices is $\Lambda_{A_0}$ such that 
$$
\displaystyle A_0=\left(\frac{\zeta_{\Lambda_1}(12)}{\zeta_{\Lambda_1}(6)}\right)^{1/3}.
$$
\end{prop}
\begin{proof}
As in the above proof, we define the function $f(r)=E_{LJ}(r\Lambda_1)$ and we compute its first derivative $f'(r)=-12r^{-13}\zeta_{\Lambda_1}(12)+12r^{-7}\zeta_{\Lambda_1}(6)$. It follows that :
$$
f'(r)\geq 0 \iff r\geq \left(\frac{\zeta_{\Lambda_1}(12)}{\zeta_{\Lambda_1}(6)}\right)^{1/6}=:r_0
$$
hence $\Lambda_{A_0}=r_0\Lambda_1$, with $\displaystyle A_0=r_0^2=\left(\frac{\zeta_{\Lambda_1}(12)}{\zeta_{\Lambda_1}(6)}\right)^{1/3}$, is the minimizer of $E_{LJ}$ among all triangular lattices.
\end{proof}
\begin{remark} We compute $A_0\approx 0.84912$, therefore the length of this lattice is $\|u\| \approx 0.99019$. Moreover we notice that $\displaystyle E_{LJ}(\Lambda_{A_0})=-\zeta_{\Lambda_{A_0}}(6)\approx -6.76425$ (it will be useful for the next part).\\ 
Because $A_0>0.63692$, Theorem \ref{ThM} is not sufficient to prove that $\Lambda_{A_0}$ is the solution of $(P)$ but a numerical investigation of $L\mapsto E_{LJ}(\sqrt{A_0}L)$ among all lattices of area $1$ seems to indicate that the solution of $(P_{A_0})$ is triangular and unique.
\begin{center}
\includegraphics[width=8cm,height=70mm]{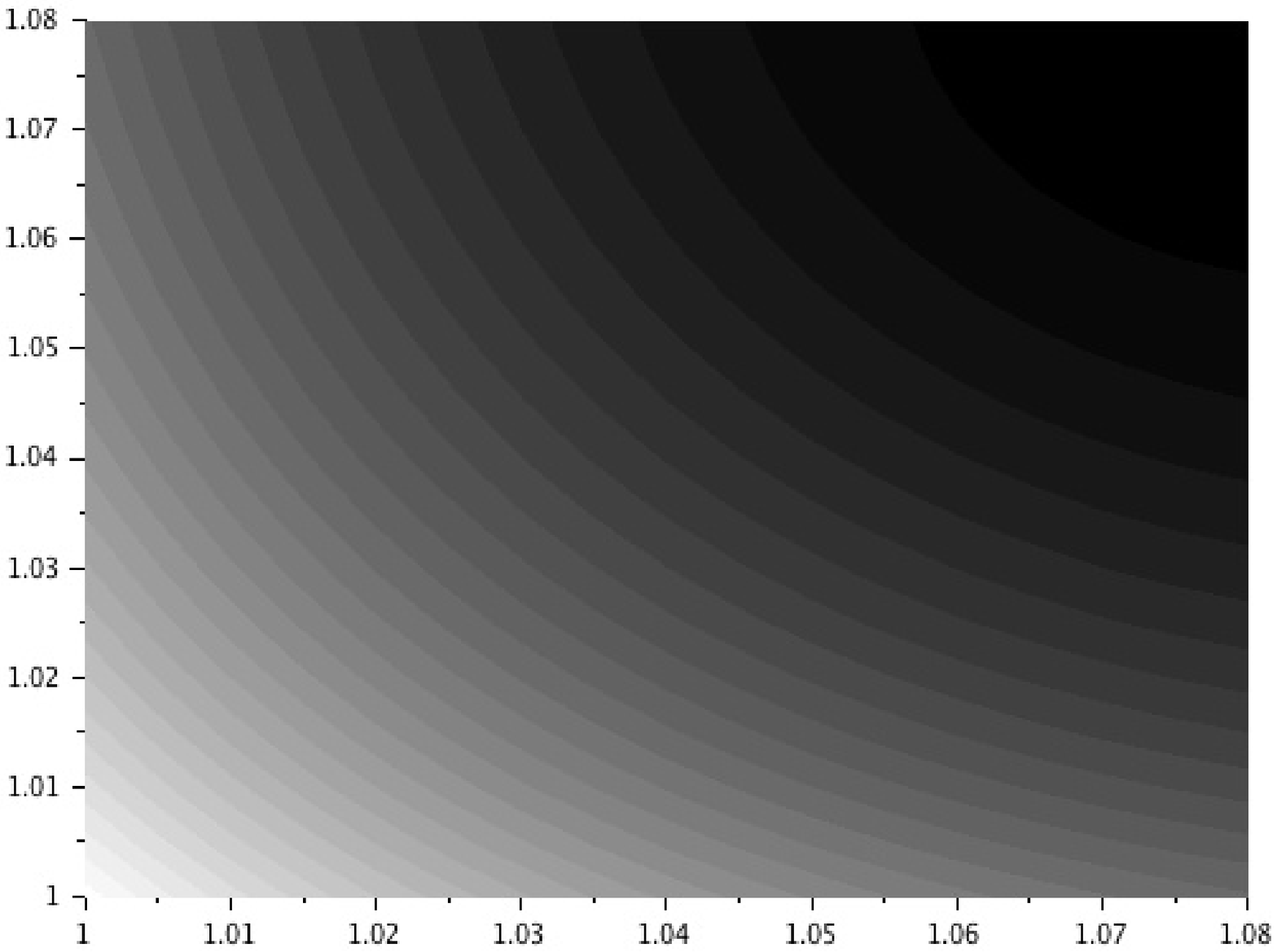}\\
\textbf{Fig. 4 : }\textit{Level sets of $(\|u\|,\|v\|)\mapsto E(\sqrt{A_0}L)$}\\
(black = minimum, white = maximum)
\end{center}

\noindent Moreover it is not difficult to prove numerically that $\Lambda_{A_0}$ is a local minimizer among all lattices. Hence we can write the following conjecture for this problem :\\ \\
\textbf{Conjecture :} \textit{The triangular lattice $\Lambda_{A_0}$ is the unique solution of $(P)$.}

\end{remark}

\subsection{Minimum length of the global minimizer}
Because our method does not show that the triangular lattice of area $A_0$ is the global minimizer of the Lennard-Jones energy among lattices, we use Blanc's proof, from \cite{BL1}, in order to find a lower bound for the minimal distance in the globally minimizing lattice. His result was for the Lennard-Jones interaction of $N$ points in $\R^2$ and $\R^3$. Xue in \cite{Xue} and Schachinger, Addis, Bomze and Schoen in \cite{26116229} improved this. We use Blanc's method because it is well suited to our problem.
\begin{prop}
If $L_0=\Z u \oplus \Z v$ is a solution of $(P)$, then the minimal distance is greater than an explicit constant $c$. Furthermore, we have $c> 0.74035$.
\end{prop}

\begin{proof}
In \cite{BL1}, Blanc proved that
$$ E_{LJ}(L_0)\geq V_{LJ}(\|u\|)-23+\frac{1}{\|u\|^{12}}\sum_{k\geq 2}\frac{16k+8}{k^{12}}-\frac{1}{\|u\|^6}\sum_{k\geq 2}\frac{32k+16}{k^6}.$$
As we have $\displaystyle E_{LJ}(L_0)\leq E_{LJ}(\Lambda_{A_0})=-\zeta_{\Lambda_{A_0}}(6)$ we obtain
$$
23-\zeta_{\Lambda_{A_0}}(6)\geq \frac{P+1}{\|u\|^{12}}-\frac{Q+2}{\|u\|^6}.
$$
with $\displaystyle P:=\sum_{k\geq 2}\frac{16k+8}{k^{12}}$ and $\displaystyle Q:=\sum_{k\geq 2}\frac{32k+16}{k^6}$.\\
Now, setting $t=\|u\|^{-6}$, we have $\displaystyle (P+1)t^2-(Q+2)t-23+\zeta_{\Lambda_{A_0}}(6)\leq 0$ which implies 
$$
t\leq \frac{Q+2+\sqrt{(Q+2)^2+4(23-\zeta_{\Lambda_{A_0}}(6))(P+1)}}{2(P+1)}
$$
and we obtain
$$
\|u\|\geq \left( \frac{2(P+1)}{Q+2+\sqrt{(Q+2)^2+4\left(23-\zeta_{\Lambda_{A_0}}(6)\right)(P+1)}}  \right)^{1/6}=:c.
$$
Since $P\approx 0.00988$, $Q\approx 1.45918$ and $\zeta_{\Lambda_{A_0}}(6)\approx 6.76425$ we get $c> 0.74035$.
\end{proof}
\begin{remark}
As we think that $\Lambda_{A_0}$ is the unique solution of $(P)$, this lower bound is the best that we can find with this method. Moreover, this bound and the second point of Proposition \ref{globmin} imply that $0.47468< |L_0|<1$.
\end{remark}

\section{The Thomas-Fermi model in $\R^2$}
In Thomas-Fermi's model for interactions in a solid, we consider $N$ nuclei at positions $X_N=(x_1,...,x_N)$, with for any $1\leq i \leq N$, $x_i\in \R^2$ , associated with $N$ electrons with total density $\rho\geq 0$. Then the Thomas-Fermi energy is given by
\begin{align*}
E^{TF}(\rho,X_N)=&\int_{\R^2}\rho^2(x)dx-\frac{1}{2}\iint_{\R^2\times \R^2}\log\|x-y\|\rho(x)\rho(y)dxdy\\
& +\sum_{j=1}^N\int_{\R^2}\log\|x-x_j\|\rho(x)dx-\frac{1}{2}\sum_{j\neq k}\log\|x_j-x_k\|.
\end{align*}
To introduce this kind of model property in quantum chemistry, refer to \cite{MathQuantique}. Because the system is neutral, the number of electrons is exactly $N$ and we study the minimization problem $\displaystyle I_N^{TF}=\inf_{X_N}\{E^{TF}(X_N)\}$ where 
$$
E^{TF}(X_N):=\inf_\rho\left\{ E^{TF}(\rho,X_N),\rho\geq 0,\rho\in L^1(\R^2)\cap L^2(\R^2),\int_{\R^2}\rho=N\right\}.
$$
By the Euler-Lagrange equations for this minimization problem, we find -- as it is explained in Section $2$ of \cite{YBLB} and Section $4$ of \cite{BLL1} -- that the minimizer $\bar{\rho}$ is the solution of 
$$
 -\Delta\bar{\rho}+\pi\bar{\rho}=\pi\sum_{j=1}^N\delta_{x_j}.
$$
It is known that the fundamental solution of the modified Helmholtz equation $-\Delta h +h=\delta_0$ -- also called ``screened Poisson equation" -- which goes to $0$ at infinity, is the radial modified Bessel function of the second kind, also called the Yukawa potential, defined in \cite{Table} and \cite{watson}, by
$$
K_0(\|x\|)=\int_0^{+\infty}e^{-\|x\|\cosh t}dt.
$$
Therefore we obtain $\displaystyle \bar{\rho}(x)=\pi\sum_{j=1}^NW_{TF}(\|x-x_j\|)$ where $\displaystyle W_{TF}(\|x\|)=\frac{1}{2}K_0(\sqrt{\pi}\|x\|)$ and finally
$$
E^{TF}(X_N)=\sum_{i\neq j}W_{TF}(\|x_i-x_j\|)+NC
$$
where $C$ is a constant independent of $N$ and $X_N$.
Now, if we consider that the nuclei are in lattice $L$, we can study, by taking the mean value of the total energy, the following energy per point 
$$
E_{TF}(L)=\sum_{x\in L^*}W_{TF}(\|x\|).
$$
\begin{remark}This potential $W_{TF}$ decreases. Therefore, it is obvious that the right problem is to minimize this energy among lattices only with a fixed area. We notice that $W_{TF}(\sqrt{.})$ is not completely monotonic on $\R^*_+$, i.e. $(-1)^n(W_{TF}(\sqrt{.}))^{(n)}(r)$ is not positive for any $n\geq 0$ and any $r>0$. Otherwise, it is explained in \cite{CohnKumar}, by using Bernstein's Theorem (see Theorem 12b of \cite{Widder}) about the following representation of a completely monotonic function $f$ 
$$
f(r)=\int_0^{+\infty}e^{-rt}d\alpha(t)
$$
where $\alpha$ is a non decreasing function, and Montgomery's Theorem \ref{Mgt} for theta functions, that the triangular lattice is the unique minimizer among lattices of $\displaystyle E_f(L):=\sum_{x\in L^*}f(\|x\|^2)$, provided we have the correct assumptions of convergence, for instance $f(r)=O(r^{-1-\eta})$ at infinity for some $\eta>0$. Nevertheless, a simple idea enables us to use theta functions  and we have the following result :
\end{remark}
\begin{thm}
$\Lambda_A$ is the unique minimizer of $E_{TF}$ among all lattices of fixed area $A$.
\end{thm}
\begin{proof}
This problem is equivalent to finding the minimizer of $\displaystyle \sum_{x\in L^*}K_0(\|x\|)$ among lattices with a fixed area. We put $y=\frac{1}{2}\|x\|e^t$ for $x\neq 0$ in the integral formula for $K_0(\|x\|)$ :
\begin{align*}
K_0(\|x\|)=\frac{1}{2}\int_{-\infty}^{+\infty}e^{-\|x\|\cosh t}dt&=\frac{1}{2}\int_0^{+\infty}e^{-\|x\|\cosh\left(\ln(2y/\|x\|)\right)}\frac{dy}{y}\\
&=\frac{1}{2}\int_0^{+\infty}e^{-y-\frac{\|x\|^2}{4y}}\frac{dy}{y}\\
&=\frac{1}{2}\int_0^{+\infty}e^{-\frac{\|x\|^2}{4y}}e^{-y}\frac{dy}{y}.
\end{align*}
Now, for any $y>0$ and any lattice $L$ of area $A$, we obtain $\displaystyle \sum_{x\in L^*}e^{-\frac{\|x\|^2}{4y}}=\theta_L\left(\frac{1}{8\pi y}\right)-1$. Hence, by Montgomery's theorem, the triangular lattice $\Lambda_A$ minimizes $\theta_L(\alpha)$ for any $\alpha>0$, and it is the unique minimizer of $L\mapsto \theta_L(\alpha)$ among all Bravais lattices with a fixed area $A$.\\
Therefore, for any $y>0$, $\Lambda_A$ is the unique minimizer of the energy $\displaystyle E_y(L):=\sum_{x\in L^*}e^{-\frac{\|x\|^2}{4y}}$ among lattices with a fixed area A. Now it is clear, because $E_y(\Lambda_A)\leq E_y(L)$ for any $y>0$ and for any lattice $L$ with area $A$, that
$$
\frac{1}{2}\int_0^{+\infty}E_y(\Lambda_A)e^{-y}\frac{dy}{y}\leq \frac{1}{2}\int_0^{+\infty}E_y(L)e^{-y}\frac{dy}{y}.
$$
Hence, for any $L$ of a fixed area $A$ : $\displaystyle E_{TF}(\Lambda_A)=\sum_{x\in \Lambda_A^*}W_{TF}(\|x\|)\leq \sum_{x\in L^*}W_{TF}(\|x\|)=E_{TF}(L)$.
\end{proof}
\begin{remark}
The Yukawa potential appears in many vortex interaction models, as the $\alpha$-model in fluid mechanics and in superconductivity (see for example \cite{Abrikosov} and \cite{11101042}). Indeed, the second author recently studied, in \cite{Zhang}, Ginzburg-Landau's model for the interactions between vortices in superconductors. He proved, by using a more general method -- that it can certainly be used for other potentials -- the same result was obtained for minimality of the triangular lattice among all lattices with fixed density. The use of results from Number Theory in Ginzburg-Landau's models for vortices can also be seen in \cite{Sandier_Serfaty}.
\end{remark}
\noindent\textbf{Acknowledgements:} We are grateful to Etienne Sandier, Xavier Blanc, Yuxin Ge and Henry Cohn for their interest and  helpful discussions. We also want to thank our colleague Chieh-Lei Wong for his remarks on the first version of this paper.
\bibliographystyle{plain}
\bibliography{biblio}

\begin{thebibliography}{10}

\bibitem{Abrikosov}
A.~Abrikosov.
\newblock {The Magnetic Properties of Superconducting Alloys}.
\newblock {\em Journal of Physics and Chemistry of Solids}, 2:199--208, 1957.

\bibitem{26116229}
B.~Addis, I.~M. Bomze, W.~Schachinger, and F.~Schoen.
\newblock {New Results for Molecular Formation under Pairwise Potential
  Minimization}.
\newblock {\em Computational Optimization and Applications}, 38:329--349, 2007.

\bibitem{BachocVenkov}
C.~Bachoc and B.~Venkov.
\newblock {Modular Forms, Lattices and Spherical Designs}.
\newblock {\em R{\'e}seaux euclidiens, designs sph{\'e}riques et formes
  modulaires}, Monographie de l'Enseignement Math{\'e}matique,
  Geneva,(37):10--86, 2001.

\bibitem{BL1}
X.~Blanc.
\newblock {Lower Bound for the Interatomic Distance in Lennard-Jones Clusters}.
\newblock {\em Computational Optimization and Applications}, 29:5--12, 2004.

\bibitem{BLL1}
X.~Blanc, C.~Le Bris, and P.-L. Lions.
\newblock {From Molecular Models to Continuum Mechanics}.
\newblock {\em Archive for Rational Mechanics and Analysis}, 164:341--381,
  2002.

\bibitem{YBLB}
X.~Blanc, C.~Le Bris, and B.~H. Yedder.
\newblock {A Numerical Investigation of the 2-Dimensional Crystal Problem}.
\newblock 2003.

\bibitem{MathQuantique}
E.~Cances, C.~Le Bris, and Y.~Maday.
\newblock {\em M{\'e}thodes Math{\'e}matiques en Chimie Quantique. Une
  introduction.}, volume~53.
\newblock Springer, 2006.

\bibitem{CohnKumar}
H.~Cohn and A.~Kumar.
\newblock {Universally Optimal Distribution of Points on Spheres}.
\newblock {\em Journal of the American Mathematical Society}, 20(1):99--148,
  January 2007.

\bibitem{Coulangeon:kx}
R.~Coulangeon.
\newblock {Spherical Designs and Zeta Functions of Lattices}.
\newblock {\em International Mathematics Research Notices}, ID 49620(16), 2006.

\bibitem{CoulLazzarini}
R.~Coulangeon and G.~Lazzarini.
\newblock {Spherical Designs and Heights of Euclidean Lattices}.
\newblock \textit{To appear in Journal of Number Theory}, 2014.

\bibitem{Coulangeon:2010uq}
R.~Coulangeon and A.~Sch{\"u}rmann.
\newblock {Energy Minimization, Periodic Sets and Spherical Designs}.
\newblock {\em International Mathematics Research Notices}, pages 829--848,
  2012.

\bibitem{DelGoethSeid}
P.~Delsarte, J.~M. Goethals, and J.~J. Seidel.
\newblock {Spherical Codes and Designs}.
\newblock {\em Geometriae Dedicata}, 6:363--388, 1977.

\bibitem{ELi}
W.~E and D.~Li.
\newblock {On the Crystallization of 2D Hexagonal Lattices}.
\newblock {\em Communications in Mathematical Physics}, 286:1099--1140, 2009.

\bibitem{Engel}
P.~Engel.
\newblock {\em {Geometric Crystallography. An Axiomatic Introduction to
  Crystallography}}.
\newblock R. Reidel Publishing Compagny, 1942.

\bibitem{Ennola}
V.~Ennola.
\newblock {On a Problem about the Epstein Zeta-Function}.
\newblock {\em Mathematical Proceedings of The Cambridge Philosophical
  Society}, 60:855--875, 1964.

\bibitem{TheilFlatley}
L.~Flatley and F.~Theil.
\newblock {Face-Centred Cubic Crystallization of Atomistic Configurations}.
\newblock {\em To appear}, 2013.

\bibitem{Rad1}
C.~S. Gardner and C.~Radin.
\newblock {The Infinite-Volume Ground State of the Lennard-Jones Potential}.
\newblock {\em Journal of Statistical Physics}, 20:719--724, 1979.

\bibitem{Table}
I.~S. Gradshteyn and I.~M. Ryzhik.
\newblock {\em {Table of Integrals, Series and Products (sixth edition)}}.
\newblock Academic Press, 2000.

\bibitem{Rad2}
R.~C. Heitmann and C.~Radin.
\newblock {The Ground State for Sticky Disks}.
\newblock {\em Journal of Statistical Physics}, 22:281--287, 1980.

\bibitem{Stef1}
E.~Mainini, P.~Piovano, and U.~Stefanelli.
\newblock {Finite Crystallization in the Square Lattice}.
\newblock {\em Nonlinearity}, 27:717--737, 2014.

\bibitem{Stef2}
E.~Mainini and U.~Stefanelli.
\newblock {Crystallization in Carbon Nanostructures}.
\newblock {\em Communications in Mathematical Physics}, to appear, 2014.

\bibitem{Mont}
H.~L. Montgomery.
\newblock {Minimal Theta Functions}.
\newblock {\em Glasgow Mathematical Journal}, 30, 1988.

\bibitem{NonnenVoros}
S.~Nonnenmacher and A.~Voros.
\newblock {Chaotic Eigenfunctions in Phase Space}.
\newblock {\em Journal of Statistical Physics}, 92:431--518, 1998.

\bibitem{Rad3}
C.~Radin.
\newblock {The Ground State for Soft Disks}.
\newblock {\em Journal of Statistical Physics}, 26(2):365--373, 1981.

\bibitem{Rankin}
R.~A. Rankin.
\newblock {A Minimum Problem for the Epstein Zeta-Function}.
\newblock {\em Proceedings of The Glasgow Mathematical Association},
  1:149--158, 1953.

\bibitem{Sandier_Serfaty}
E.~Sandier and S.~Serfaty.
\newblock From the {G}inzburg-{L}andau {M}odel to {V}ortex {L}attice
  {P}roblems.
\newblock {\em Communications in Mathematical Physics}, 313(3):635--743, 2012.

\bibitem{SarStromb}
P.~Sarnak and A.~Str{\"o}mbergsson.
\newblock {Minima of Epstein's Zeta Function and Heights of Flat Tori}.
\newblock {\em Inventiones Mathematicae}, 165:115--151, 2006.

\bibitem{11101042}
C.H. Sow, K.~Harada, A.Tonomura, G.~Crabtree, and D.~G. Grier.
\newblock {Measurement of the Vortex Pair Interaction Potential in a Type-II
  Superconductor}.
\newblock {\em Physical Review Letters}, 80:2693--2696, 1998.

\bibitem{Terras}
A.~Terras.
\newblock {\em {Harmonic Analysis on Symmetric Spaces and Applications}},
  volume~1.
\newblock Springer-Verlag, 1985.

\bibitem{Crystal}
F.~Theil.
\newblock {A Proof of Crystallization in Two Dimensions}.
\newblock {\em Communications in Mathematical Physics}, 262(1):209--236, 2006.

\bibitem{Venkov1}
B.~Venkov.
\newblock R{\'e}seaux et designs sph{\'e}riques.
\newblock {\em R{\'e}seaux euclidiens, designs sph{\'e}riques et formes
  modulaires}, Monogr. Enseign. Math., Geneva,(37):10--86, 2001.

\bibitem{VN1}
W.J. Ventevogel and B.R.A. Nijboer.
\newblock {On the Configuration of Systems of Interacting Particle with Minimum
  Potential Energy per Particle}.
\newblock {\em Physica A-statistical Mechanics and Its Applications}, 92A:343,
  1978.

\bibitem{VN2}
W.J. Ventevogel and B.R.A. Nijboer.
\newblock {On the Configuration of Systems of Interacting Particle with Minimum
  Potential Energy per Particle}.
\newblock {\em Physica A-statistical Mechanics and Its Applications},
  98A:274--288, 1979.

\bibitem{VN3}
W.J. Ventevogel and B.R.A. Nijboer.
\newblock {On the Configuration of Systems of Interacting Particle with Minimum
  Potential Energy per Particle}.
\newblock {\em Physica A-statistical Mechanics and Its Applications},
  99A:569--580, 1979.

\bibitem{watson}
G.~N. Watson.
\newblock {\em {A Treatise on the Theory of Bessel Functions}}.
\newblock Cambridge University Press, 1922.

\bibitem{Widder}
D.~V. Widder.
\newblock {\em {The Laplace Transform}}.
\newblock Princeton University Press, 1946.

\bibitem{Xue}
G.~L. Xue.
\newblock {Minimum Inter-Particle Distance at Global Minimizers of
  Lennard-Jones Clusters}.
\newblock {\em Journal of Global Optimization}, 11:83--90, 1997.

\bibitem{Zhang}
P.~Zhang.
\newblock {On the Minimizer of Renormalized Energy related to Ginzburg-Landau
  Model}.
\newblock (submitted), 2014.

\end{thebibliography}
\vspace{3mm}
\noindent LAURENT B\'{E}TERMIN, e-mail: \textbf{laurent.betermin@u-pec.fr} \\
\\
\noindent PENG ZHANG, e-mail : \textbf{peng.zhang@univ-paris-est.fr} \\

\noindent UNIVERSIT\'{E} PARIS-EST CR\'{E}TEIL,\\
LAMA - CNRS UMR 8050,\\
61, Avenue du G\'{e}n\'{e}ral de Gaulle, 94010 Cr\'{e}teil. France\\

\end{document}